\documentclass[12pt,a4paper,reqno]{amsart}
\usepackage{inputenc}
\usepackage{amssymb}
\usepackage{amscd}
\usepackage{enumitem}

\newlist{notes}{enumerate}{1}
\setlist[notes]{label=Note: ,leftmargin=*}
\usepackage{enumitem}
\usepackage{graphicx}
\usepackage{siunitx}
\usepackage{tikz-cd}
\usepackage{bm}
\numberwithin{equation}{section}

\usepackage{mathtools}
\usepackage[tableposition=top]{caption}
\usepackage{booktabs,dcolumn}

\DeclareFontFamily{OT1}{rsfs}{}
\DeclareFontShape{OT1}{rsfs}{n}{it}{<-> rsfs10}{}
\DeclareMathAlphabet{\mathscr}{OT1}{rsfs}{n}{it}

\theoremstyle{plain}

\newtheorem{theorem}{Theorem}[section]

\newtheorem{lemma}[theorem]{Lemma}

\theoremstyle{definition}

\newtheorem{definition}[theorem]{Definition}

\begin{document}

\title[EUCLIDEAN RHYTHM]{EUCLIDEAN RHYTHM WITH PALINDROMIC RESTS}

\author{Paraj Mukherjee}

\begin{abstract} The structure of the Euclidean algorithm can be used to generate a very large family of rhythms. In this paper, we explore a very specific family of Euclidean rhythms, the Euclidean rhythms in which the rests have a palindromic structure. We look at the structural and geometric properties of such rhythms; most of the properties have a certain combinatorial interest. We also show how operations can be defined on such rhythms to construct a larger family of rhythms. Towards the end of the paper, we also show applications of these rhythms in contemporary Indian Classical music.
\end{abstract}       

\maketitle


\section{Introduction}
Given \(n\) time intervals and \(p<n\) onsets, we try to distribute the onsets as evenly as possible among the intervals. The very nature of the Euclidean algorithm makes the Euclidean rhythms have a very interesting property: the distance between all pairs of onsets gets maximised. We may represent each rest by \(0\) and each onset by \(1\). The problem then becomes constructing a binary sequence of \(n\) bits with \(p\) 1's such that the 1's are spread as evenly as possible among the 0's \cite{toussaint}. The generated rhythm is denoted by \(E(p,n)\). Let \(k=n-p\) be the number of rests. If \(p \big| n\), then we are done! The number of rests will be equal between any two consecutive onsets. For example, if \(p=4,n=16\), \(E(p,n)\) becomes \([\times \cdot \cdot \cdot \times \cdot \cdot \cdot \times \cdot \cdot \cdot \times \cdot \cdot \cdot \times]\) (In this representation, \(\times\) denote the onsets and \(\cdot \) denote the rests. The last onset in this representation corresponds to the first onset of the next cycle). As all the Euclidean rhythms have this notion of \(evenness\), it is perhaps very natural to ask what \(E(p,n)\) would look like if the rests had a palindromic structure.\footnote{Palindromes are the most symmetric numbers. Both the concepts of "evenness" and "symmetry", when put together, result in structures that are aesthetically  'as uniform as possible'.} The problem of primary interest is when \(k\) and \(n\) are relatively prime numbers. The Euclidean rhythms can be generated by many methods, like some modification of Bresenham's line drawing algorithm using some matrix constructions \cite{morill}. In this paper, we use the Bjorklund's algorithm \cite{bjork100}. If some pattern of digits occurs more than twice, we refer to them as units. In each step, we have two parts: the left side consists of some patterns that repeat themselves (units), and on the right, we have the remainder units, which we distribute among the patterns on the left in each step. So initially, we have \(p\) `1's and \(k\) `0's side by side. The process continues till no remainder units are left or one remainder unit is left. At the end of it, we are left with the left side forming the majority of the string and consisting of recurssive, evenly spread onsets \cite{bjork99}. \\

\newpage
\section{Preliminaries and notations} 
The rhythm generated by \(p\) onsets and \(k\) rests is denoted by \(E(p,n)\), where \(p+k=n\). To every corresponding \(E(p,n)\), \(V(p,k)\) denote the number where each digit is the number of rests between the corresponding onsets. The \(V(p,k)\) representation of any rhythm can be written in general. \\
A rhythm of \(n\) time intervals can be represented by  a subset of \(\{0,1,...,n-1\}\) which represent the pulses that are onsets. We can use the additive group \(\mathbb{Z}_n\) to represent any rhythm. Any rhythm can be written as \(R=\{A_j: j \in \{0,1,2,...n-1\}\}\) where the onsets occur at the \(A_j\)th pulses. To represent it geometrically, we place \(n\) points uniformly around the circumference of a circle and then label them clockwise. For any two points \(A_i, \ A_j\) on the circle we can define distance \(\delta\) between them in two senses, clockwise and anti-clockwise which are \(j-i\) and \(|j-i \ (mod \ n)|\) respectively. We define the least of these two distances \(\delta_1\) as the chordal distance, which belongs to the set \(\{0,1,...,[\frac{n}{2}]\}\). The multiplicity of \(\delta\) and \(\delta_1\) is defined as the number of distinct pairs in the rhythm that have the distances \(\delta\) and \(\delta_1\) respectively, between them.  \\
A rhythm is called Erdos-Deep if for every multiplicity \(\{1,2,...,p-1\}\) there exists one distance with that multiplicity. 
A rhythm can be rotated by an integer \(l \geq 0\), the new rhythm is \(R_1=\{A_{(i+l) \ mod \ n} : A_i \in R\}\). A rhythm can be scaled or multiplied by an integer \(l\geq 1\), \(R_2 =\{A_{li} : A_i \in R\}\) and is of time interval \(ln\).

\begin{definition}
\(k,p \in \mathbb{Z}\) such that \(k>p\)

\begin{enumerate}
   \item if \(p=0\), \(f(k,p)=k\)
   \item if \(p>0\), \(f(k,p)=f(p,k \ mod \ p)\)
\end{enumerate} 
 
\end{definition}
Bjorklund's algorithm follows exactly these recursive steps. It is very clear that this has the same structure as the Euclidean algorithm.  \\
We also state a well-known result by Dirichlet on prime numbers in arithmetic progressions, which will be used once later on: An \(AP\) in which the first member and the common difference
are pairwise prime integers has infinitely many primes in it. \\
Two rhythms are called homometric if they have the same chordal distance histogram, i.e., if a chordal distance \(\delta\) occurs \(e\) times in one, it also occurs \(e\) times in the other one as well. If \(P\) and \(Q\) are homometric, two vertices (onsets) \(p_i\), \(q_i\) on \(P\), \(Q\) respectively, are called isospectral if they have the same histogram of chordal distances to all other vertices in their respective rhythms. \\
In Section \(10\) we show a few applications of Euclidean rhythms with palindromic rests in Indian Classical music. One needs basic knowledge of Indian Classical music to understand the compositions.

\section{Evenly-spread rhythms}
Euclidean rhythms are a type of rhythm that is characterised by an even distribution of onsets across a given time interval. These rhythms are created by taking a number of beats and distributing them as evenly as possible across the time interval. \\
We first describe what an almost-evenly spread rhythm is. We say a rhythm is `almost evenly spread' if the difference between the chordal distances between every consecutive onset is `not very large'. This is obviously not a very mathematical definition. Given two almost-evenly spread rhythms, one can compare them; if the sum of digits of both the \(V(p,k)\)s is the same, the one with more digits is more evenly spread. For a given \((p,k)\) the most evenly spread rhythm is unique (i.e., the digits of the \(V(p,k)\) representation is unique). For maximum evenly-spreadness, the sum of all possible chordal distances has to be the maximum, i.e. \(\sum_{0 \leq i < j \leq n-1}|j-i \ (mod \ n)| \) is maximum. There are also palindromic rhythms, which are almost evenly spread but not Euclidean. Out of all these various most-evenly spread rhythms, we consider the ones where \(V(p,k)\) is a palindrome. 
\begin{notes} 
    \item It may so happen that the rhythm we get by Bjorklund's algorithm has no palindromic structure, but we can get one by considering some rotation of it. That rhythm is also an Euclidean rhythm by definition. In the next section, we find the rhythms that have a palindromic structure using Bjorklund's algorithm; the solution set can be expanded by considering those rotated rhythms as well. We will continue this discussion a bit more in Section 7.
\end{notes}

\section{The Euclidean Rhythms with Palindromic rests} 
We are given \(p\) `1's and \(k\) `0's, \(p < k\) and \(gcd(p,k)=1\). In that case, one remainder unit is left in the end. We demonstrate when the number representing the rests will be a palindrome. We denote the set of all such rhythms by \(\mathcal{T}\).     
       The problem can be broken in two cases. 
       
\begin{enumerate} 
    \item \(2p>k>p\); i.e., the acting direction of the algorithm will change right after \(1\) step
    \item \(k>2p,k>3p,.... k>[\frac{k}{p}]p\); i.e., the acting direction of the algorithm will change after some finite steps greater than \(1\)
    
 \end{enumerate}
In second case, we have \[1\underbrace{000....}_{[\frac{k}{p}]-1} 1000.... 1000........  \underbrace{0000....}_{k-([\frac{k}{p}]-1)p}  \]  

Now in this structure, \(2p>k-([\frac{k}{p}]-1)p>p\) so the acting direction changes after \(1\) step. Therefore, the structure in the second case is isomorphic to one in the first case (each \(1000...\) can be thought of as a \(1\)). Let \(P'\) be any structure in second case, which is isomorphic to \(L'\) of first case. For any \((p,k)\), the structure itself follows case \(1\) or is isomorphic to one following case \(1\). Due to structural isomorphism, if \(L' \in \mathcal{T}\), then \(P' \in \mathcal{T}\). We call such a \(L'\) a base.  On the other hand, we also see that if \(P' \in \mathcal{T}\), then \(L' \in \mathcal{T}\). From this one-to-one mapping see that our problem gets reduced to case 1 only, and solutions for case 2 can be found accordingly. In the first case, the algorithm can go on for an arbitrary number of steps, depending on the values of \(k\) and \(p\) and the acting direction may again change arbitrarily many times after the second step. \\
\begin{lemma}
The number representing the rests for values of \((p,k)\) for which the algorithm carries out 3 steps and does not change the acting direction after initially changing in the second step will be a palindrome. 
\end{lemma}
\begin{proof}

The fact that \((p,k)\) for which the condition holds will correspond to a solution can easily be seen by establishing a relation between \(k\) and \(p\). After the first step, there are \(p\) `\(10\)'s and \(k-p\) `\(0\)'s. After the second step, there are \(k-p\) `100's and \(2p-k\) `10's, and as the direction remains the same, \(2k>3p\). In the third step, there are \(2p-k\) `10010's and \(2k-3p\) `100's. Since the algorithm terminates here, we have
\begin{equation}
\label{2.1} 2k-3p=1 
\end{equation}
Therefore \((p,k)\) satisfying \ref{2.1} corresponds to a solution set for our problem. By calculating using Bjorklund's algorithm, we see that \(p=1,3\) does not give us a solution. Hence, \((p,k)\) of the solution set is  
\begin{equation}
 \{(p,k) \ | \ 2k-3p=1 ;p \neq 1,3\}
\end{equation} \\
The first few values of \((p,k)\) are \((5,8),(7,11),(9,14),(11,17),(13,20)\) and so on. The structures of these base Euclidean rhythms can be easily found out by doing the algorithm itself. The smallest rhythm of this family which is found for \(p=5, \ k=8\) looks like: \([\times \cdot \cdot \times \cdot \times \cdot \cdot \times \cdot \times \cdot \cdot \times]\) with \(V(p,k)=21212\). One can easily see that \(V(p,k)\) of the subsequent other base rhythms can be written by just adding a `\(12\)' at the end of the \(V(p,k)\) representation of previous base rhythm. Let the set of these base rhythms be set \(A\). \(V(p,k)\) representation of any element looks like:
\begin{equation}
A=\underbrace{21212.....2}_{some \ finite \ times}
\end{equation}
\end{proof}
We now consider two subcases under case \(1\)- 
\subsection{subcase \(1\)}Let the acting direction change in the third step, i.e., if \(3p>2k\). If the algorithm terminates after two more steps, the left part of the number starts as `1001010' but the right side ends as `10010' (which is the remainder unit). Hence, we have no solution. If the algorithm goes on for arbitrary finite steps, changing direction arbitrary many times, we get no solution. This is because the right side will always have `10' at the end, which does not make the rests have a palindromic structure. 
\subsection{subcase \(2\)}Let the algorithm not end after three steps. In this case, let the acting direction not change again in step \(4\). If the process terminates at the next step, we get no solution, as the left part of the number starts as `10010100' but the right side ends as `10010'. After step \(4\), the process can go on for arbitrary finite steps, with the direction changing arbitrarily many times. In each possibility, one can see that no solution can be found; every time either the end digits of \(V(p,k)\) representation do not match, or even if they do match, the number is not a palindrome. Now let the acting direction change in step \(4\). If the process terminates in the next two steps, we have solutions! In that case, \(V(p,k)\) representation looks like \[\underbrace{2122212....2122212}_{some \ finite \ times}\] One can see that \(V(11,30)=21222122212\), 
\(V(15,41)=212221222122212\) and so on. Observe that \(V(p,k)\) of the subsequent other base rhythms can be written by just adding a `2212' at the end of the \(V(p,k)\) representation of the previous base rhythm. Let the set of these base rhythms be set \(B\). If the process does not terminate in the next two steps, one can look at the subsequent cases and see that no more solutions can be found.

We can now generate the whole solution set by considering case 2. For example, from \(E(5,13)\) we can generate \(E(5,18), \ E(5,23), \ E(5,28) \) and so on, \(E(5,18)=[\times \cdot \cdot \times \cdot \times \cdot \cdot \times \cdot \times \cdot \cdot \times \cdot \times \cdot \cdot \times] \). In general, when considering the base rhythms of set \(A\), \(V(p,k)\) generally looks like
    \[[\underbrace{(m+1)m(m+1)m....(m+1)}_{arbitrary \ many \ times}], \ m \ge 1\] 
Let the set of these rhythms and the base rhythms of \(A\) be set \(\tau_1\).
Similarly, by considering set \(B\), \(V(p,k)\) looks like \[[\underbrace{(m+1)m(m+1)(m+1)(m+1)m(m+1)....(m+1)m(m+1)}_{arbitrary \ many \ times}], \ m \geq 1\]
Let the set of these rhythms and the base rhythms of \(B\) be set \(\tau_2\). From \(V(p,k)\) we thus get the whole family \(\mathcal{T}\) of Euclidean rhythms in which \(V(p,k)\)s are palindromes. \(\mathcal{T}= \tau_1 \cup \tau_2\). \\ 
Due to slightly more simplicity in its structure, we study \(V(p,k)\) instead of the rhythms themselves; all the results analogously apply for rhythms. The process by which this family of rhythms is generated holds little importance to any composer or musician; however, the sheer symmetry and overall geometry of distances do make these rhythms an attractive choice in any musical piece. The beauty or aesthetics of any rhythm depend mostly on its mathematical structure, which in turn depends on the ability of one's mind to see these mathematical structures subconsciously in the musical piece. Therefore, we study a few properties of \(\mathcal{T}\). The proofs of most of the results discussed hereafter hold more importance than the results themselves. Many of these properties may be quite elementary, but the overall beauty of any rhythm depends on all of them.

\section{Erdos-Deep}
The result that any element of \(\mathcal{T}\) is Erdos deep follows from the result that any Euclidean rhythm where \(gcd(p,n)=1\) is Erdos-deep. Nevertheless, we demonstrate a proof that shows the beautiful geometry of chordal distances and multiplicities in elements of \(\mathcal{T}\). The pattern of distances and the multiplicity pattern they follow are quite interesting from a musician's point of view.
\begin{theorem}
    Every element of \(\mathcal{T}\) is an Erdos-deep rhythm.
\end{theorem}
\begin{proof}
We illustrate the proof for the set \(A\), the proof for \(\tau_1\) follows similarly. Let \(x \in A\). The set of multiplicities is - \(\{1,2,...,\frac{p-1}{2},\frac{p+1}{2},...,p-1\}\). The distance 6, i.e., [33] is unique and can occur in only 1 way. Through some simple combinatorial counting of \(\delta\) by writing the \(V(p,k)\) of \(x\), one can see that the distance 2, i.e., [2] occurs \(\frac{p-1}{2}\) times, and the distance 3, i.e., [3] occurs \(\frac{p+1}{2}\) times. Continuing this counting, we can see that the distance 5, i.e., [23] occurs \(p-1\) times, distance 7, i.e., [232] occurs \(\frac{p-3}{2}\) times, distance 10, i.e [2323] occurs \(p-2\) times. In general, distance [2323...23] with number of `2's \(=\) number of `3's \(= q\) (let) satisfying \(p-q > \frac{p+1}{2}\) occurs \(p-q\) times and for [2323...2] with number of `2's \(=m\), number of `3's \(=m-1\) satisfying \(\frac{p-2m-1}{2}>1\) occurs \(\frac{p-2m-1}{2}\) times. Hence, we exhaust the whole set of multiplicities in this process. We can do a similar analysis of the relations between distances and
multiplicities for elements of \(B\) and extend that to \(\tau_2\). This completes the proof.
\end{proof}

\section{Almost Winograd-Deep}
As introduced by Winograd in a 1966 class project report \cite{winograd}, we call a rhythm \(Winograd-deep\) if every chordal distance \(1,2,...,[\frac{n}{2}]\) has a
unique multiplicity. Now one of the structural difficulties we have with elements of \(\mathcal{T}\), is that several distances do not occur in the rhythms. For example, in elements of set \(A\), distances like \(1,4,9\) are not possible. So no element of \(\mathcal{T}\) is \(Winograd-deep\). We define a rhythm \(almost \ Winograd-deep\) if the chordal distances in the set \(\{1,2,...[\frac{n}{2}]\}\) have a unique multiplicity; the occurance of every chordal distance is not necessary. \\
\begin{theorem}
    Every element of \(\mathcal{T}\) is almost Winograd-deep.
\end{theorem} 
\begin{proof} 
Let \(x \in \tau_1\). We have proved in the previous section that \(x\) is Erdos deep. The available set of multiplicities is - \(\{1,2,...,k-1\}\). Let \(D_1,D_2,...D_{k-1}\) be the distances we used in section 3. Some of these are chordal distances while some of them are not. If \(D_j\)s are chordal distances, we can see easily that the chordal distance associated with any multiplicity is unique. We name it set \(E\). Let \(D_i\)s be not chordal. If \(D_i\) is not chordal, we have \(n-D_i <[\frac{n}{2}]\) so \(n-D_i\) becomes chordal. The multiplicity of \(D_i=\) multiplicity of \(n-D_i=l\) (let), then \(d-D_i\) is the chordal distance associated with \(l\). We name the set of all \(n-D_i\)s as \(F\). As \(E \cap F = \phi\), we have all the distances occurring in \({1,2,...[\frac{n}{2}]}\) having unique multiplicities. The proof for any element of \(\tau_2\) follows similarly. 
\end{proof} 
In the process, we have found another set of distances (which is the set of all possible chordal distances in \(x\)) that can be used to prove that \(x\) is Erdos-deep. Let \(G=\) set of all chordal distances of \(x\), \(H=\) set of all multiplicities, and we consider the natural mapping \(f:G \rightarrow H\) where \(g_1\) is mapped to \(h_1\) if \(g_1\) occurs \(h_1\) times. As \(x\) is almost-Winograd deep, \(f\) is one-one. Also, since \(f\) is Erdos-deep, \(f\) is surjective. Therefore, \(f\) is a bijective mapping. We sum all this up and state the following theorem, which is the heart of the last two sections. 
\begin{theorem}
    For every \(x \in \mathcal{T} \) there exists a bijective map between the set of all chordal distances and the set of all multiplicities.
\end{theorem}

\section{Homometric Rhythms}
\begin{figure}[htp]
    \centering
    \includegraphics[scale=0.45]{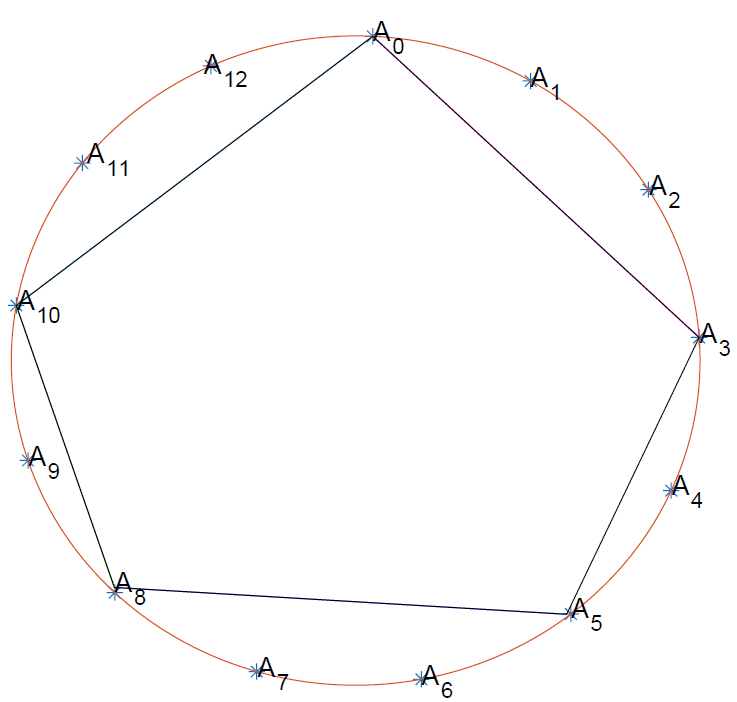}
    \caption{\(E(5,13)\)}
\end{figure}
 For any rhythm, all the rotated structures are trivially homometric to one another. It is very easy to prove that these are the only homometric rhythms for any member of \(\mathcal{T}\).  
\begin{theorem}
    For every member \(x\) of \(\mathcal{T}\) the rotated rhythms are the only ones homometric to \(x\). 
\end{theorem}
\begin{proof}
We prove it for any element of \(A\), result follows similarly for \(\tau_1\). Let \(x \in A\) with \(p\) onsets. Then, chordal distances \(3\) and \(2\) occurs \(\frac{p+1}{2}\) and \(\frac{p-1}{2}\) times respectively. Let \(y\) be a rhythm homometric to \(x\). Distance \(3\) can be formed in the pattern \([3]\) or \([2+1]\). The second case is not possible in \(y\) as chordal distance \(1\) is absent in \(x\). Therefore, each chordal distance \(3\) is formed as \(2\) rests between two consecutive beats. Similarly, we see that chordal distance \(2\) is formed as \(1\) rest between two consecutive beats. So we need to distribute \(\frac{p+1}{2}\) pairs of onsets such that the distance between them is 3 and \(\frac{p-1}{2}\) pairs of onsets such that distance between them is \(2\). One can clearly see that the original rhythm \(x\) or some of its rotated forms are the only possibilities here. Proof for elements of \(\tau_2\) follow similarly. Hence, we have proved our claim. 
\end{proof} 

This result is also easy to guess by intuition, considering that the Euclidean rhythm is one unique rhythm for any \((p,k)\). If we speak about isospectral vertices, each rhythm is obviously homometric to itself, and by the symmetry of all the rhythms along the vertical axis, each pair of the symmetric vertices is thus isospectral to each other. For example in figure \(1\), vertices \((A_3, A_{10}),(A_5, A_8)\) are pairwise isospectral. Now if we consider the rotated rhythms whose one vertice lies on \(A_0\), we can take the isospectral vertices accordingly in the same way. 
We now develop the observations made in Section 3 even further. \\ 

For any \(x \in A\), all the possible rhythms formed by rotations are evenly spaced. Total number of combinations possible by taking digits of \(V(p,k)=\) the structures where chordal distance \(2\) occurs in consecutive pairs of onsets \(+\) structures where chordal distance \(2\) does not occur in consecutive pairs of onsets (total number of such structures \(=p\), it is the same as the number of ways \(\frac{p-1}{2}\) `1's can be distributed among \(p\) places so that two `1's are not side by side). All these rhythms are the only homometric rhythms of \(x\) that have an onset at the first pulse. Also, \(V(p,k)\) representation of each of these structures can be made a palindrome by giving a particular rotation! This matches the fact that these are homometric to \(x\), we can say that an isomorphism exists between each of them. All of the \(p\) rhythms can be obtained from \(x\) by suitable transformations. We select a string of consecutive digits of \(V(p,k)\) starting from left digit. For each such selection, a symmetric selection from the right digit can also be made. Observe that by swapping any string (reversing the digits of the string in \(V(p,k)\)) we get one homometric rhythm of \(x\). Strings with \(2\) as the extreme right digit (when taking selections starting from the left digit) give us the same \(V(p,k)\), so we do not consider them. In \(E(p,n)\), there are \(p\) beats, total number of `1's in \(V(p,k)=\frac{p-1}{2}\), total number of strings \(=\frac{p-1}{2} \times 2=p-1\), swapping each of which gives us one rhythm homometric to \(x\). One can also observe this phenomenon by looking at its geometric representation. For example, in figure \(1\), \(A_3 \rightarrow A_2, \ A_{10}
\rightarrow A_{11}, \ A_5 \rightarrow A_6\) and \(A_{10} \rightarrow A_{11}, \ A_8 \rightarrow A_7\) and \(A_3 \rightarrow A_2 \) swappings give us all the homometric rhythms that have an onset at the first pulse. Like this, such swappings can be studied for more complicated rhythms. 

These rhythms can also be looked upon as a group action of \(S_n\) on \(x\). Let \(f\) be the group action. Then \(A_3 \rightarrow A_2\) swapping can be looked upon as a group action of the cycle \(S=(3 \ 4)\) on \(x\). \(f\) is defined as : \[\{A_j: j \in \{1,2,...,13\}\} \rightarrow \{A_{S(j)}: j \in \{1,2,...,13\}\}\] 

\section{Constructing new rhythms}
In the previous section, we saw how any pair of isospectral vertices look for any rhythm of \(\mathcal{T}\) and any of its homometric pair. An  operation which involves isospectral vertices is illustrated below: 
\begin{definition}[Pumping \cite{pumping}] Let \(P_1\), \(P_2\) be a homometric pair corresponding \(E(p,n)\) with isospectral vertices \(p_x, \ p_y\) respectively. \(S=\{+s_i, -s_i : i \in \{1,..,r\}\}\). We define a \((m,r,S)\) pumping of \(P_1, P_2\)  \((m \geq 1, r \geq 0)\) as a new pair of rhythms with number of time intervals\(\ =n'\)  and number of onsets \(=k'\) with \(n'=mn\) and \(k'=k+2r\). These are formed by replacing \(p_x\) in \(P_1\) with \(\{p_x, p_{x+s_i}, p_{x-s_i}\}\) and \(p_y\) in \(P_2\) with \(\{p_y,p_{y+s_i}, p_{y-s_i}\}\)    
\end{definition} 

\begin{figure}[htp]
    \centering
    \includegraphics[height = 9.5cm, width = 10cm]{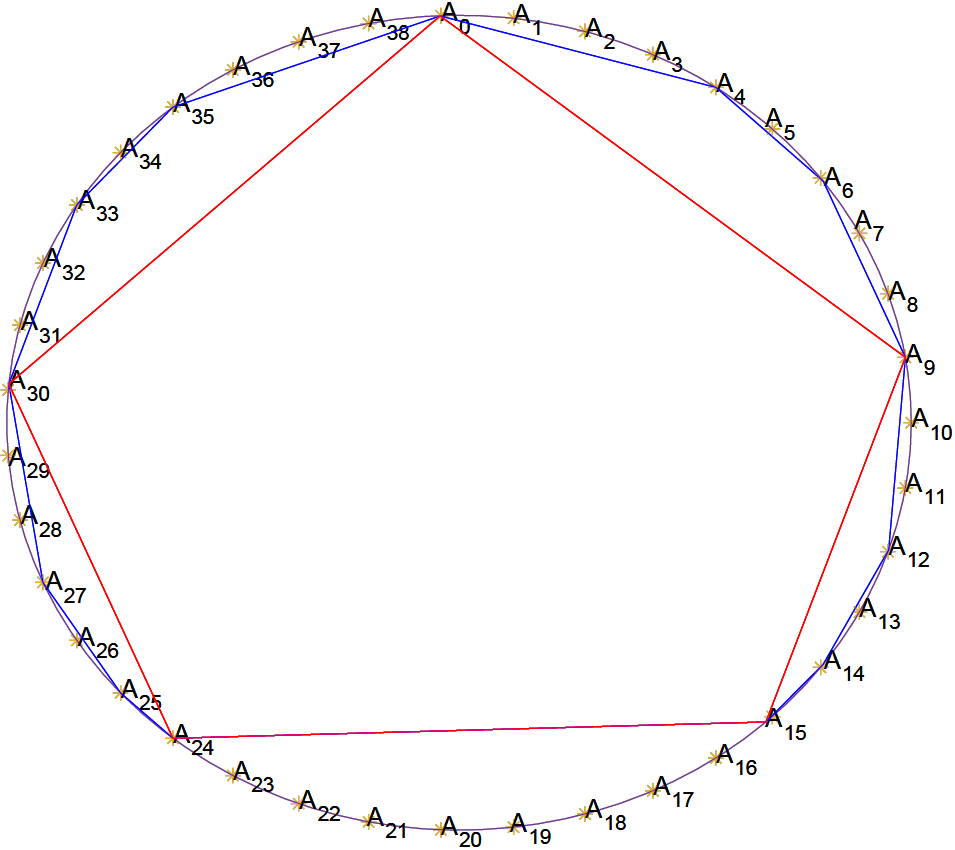}
    \caption{union of all points after \((3,2,S)\) pumping applied to \(E(5,13)\) and itself (taking itself as the homometric rhythm), where 
    \(S=\{+3,-3,+5,-5\}\)}
\end{figure}

It is quite apparent that any pair of rhythms formed by any pumping will be homometric to one another, and one can be obtained by giving the other one a suitable rotation. For the sake of constructing new rhythms, we can also apply different pumping to a pair of homometric rhythms; they would not remain homometric anymore. We can also apply pumping to a separate pair of isospectral points and take the union of all points to form a new rhythm.  \\
Pumping two rhythms is a way to systematically create new rhythms from \(\mathcal{T}\). Like this, we can use various other operations and algorithms to construct new rhythms. The various properties of them can be studied accordingly with appropriate tools. One should remember that we should always try to create these algorithms in such a manner that the connection between the newly formed rhythm and the rhythms from which it is formed is understandable while practically performing any musical composition. Some of the basic ways of doing this are illustrated: 
\begin{itemize}
    \item Addition: We can define addition between two rhythms in the most natural way: For rhythms \(P_1(V(p,k_1))\) and \(P_2(V(p,k_2))\), the palindrome number of \(P_1+P_2\) is defined as \(V(p,k_1+k_2)\) and the rests are distributed by component addition. Subtraction can be defined accordingly. \
    \item Addition by changing the component: If two rhythms have different numbers of onsets in general, one can define addition, which is not the usual component addition, but the components get adjusted sideways while doing the addition. One can see that multiple such adjustments are possible for any two rhythms. For example, \(V(p,k)\) representation of \(E(7,11)+E(5,18)\) can be written as \(2555552\) by arranging the components in a symmetric way about
    each other and performing addition. Subtraction can be defined accordingly.
    \item Introducing a set of onsets: We can introduce a set of onsets in a rhythm and distribute them in the rests of the rhythm in a definite manner. Let \(S\) be a set of onsets, \(P\) be a rhythm. \(P_S\) be the rhythm formed after this operation, then \(P_S\) is defined as the rhythm formed by adding the elements of \(S\) in \(P\) in corresponding positions (the process is quite similar to that of pumping). 
    \item Breaking the rhythms: We can break a rhythm to form two; the breaking is done such that there is a beat in the first rhythm formed at the end. These two can now be used accordingly in any composition. For example, \(V(5,8)\) can be broken down as \([\times \cdot \cdot \times \cdot \times \cdot \times]\) and \([\cdot \times \cdot \times \cdot \cdot \times]\).
    \item Taking the complement of a rhythm - We can take the complement rhythm for any rhythm of \(\mathcal{T}\), the onsets in the complement rhythm will have a palindromic structure. 
    \item Concatenation \cite{conc} - We can also take the concatenation of two rhythms of \(\mathcal{T}\). For example, \(E(5,13) \bigoplus E(5,18)\)= \[[\times \cdot \cdot \times \cdot \times \cdot \cdot \times \cdot \times \cdot \cdot \times \cdot \cdot \cdot \times \cdot \cdot \times \cdot \cdot \cdot \times \cdot \cdot \times \cdot \cdot \cdot \times]\]
    \item We can consider Euclidean rhythms with a different number of onsets and the same number of time intervals as an element of \(\mathcal{T}\) and use them accordingly in some compositions; there are definite mathematical relations between them and the element of \(\mathcal{T}\). 
\end{itemize} 

\begin{figure}[htp]
    \centering
    \includegraphics[scale=0.5]{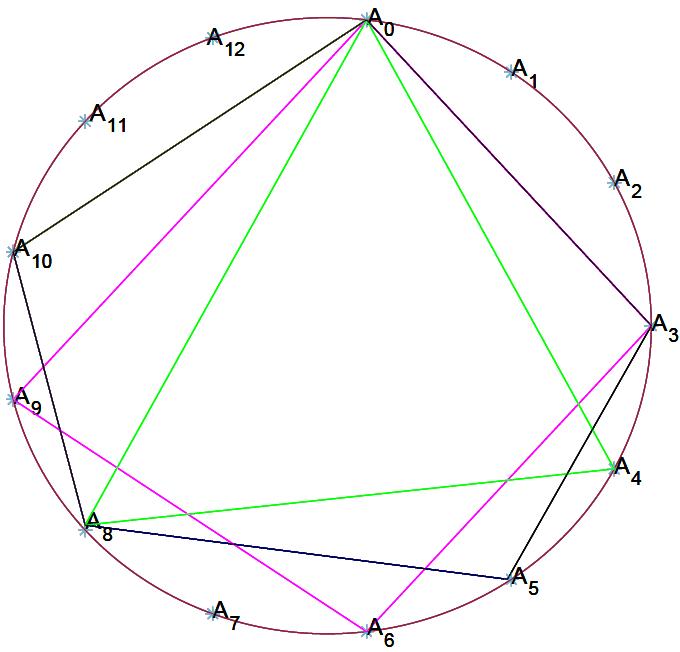}
    \caption{some other euclidean rhythms on \(13\) intervals} 
    \label{fig:galaxy}
\end{figure}
We can extend some of the ideas. After pumping a rhythm and itself and taking union of both the rhythms (like in figure \(2\)), we can divide the figure vertically (resulting in dividing the rhythm symmetrically), and then consider a rhythm homometric to one half; the corresponding rhythm homometric to the other half will be symmetric to the first one. We can then take the union of the points and form another rhythm; further properties of such rhythms can be studied. This can be done by first studying what the homometric rhythms of any arbitrary rhythm look like. One can explore this by taking the necessary ideas from Patterson's paper in crystallography\cite{patt}.

\section{Arithmetic Progressions} 
The \((n,p,k)\) values of the solution set \(\tau_1\) can all be written in matrix form. The matrix of the \(n\) values is:
\begin{equation}
\begin{pmatrix}
13 & 18 & 23 & 28 & 33 & ... & ...\\
18 & 25 & 32 & 39 & 46 & ... & ...\\
23 & 32 & 41 & 50 & 59 & ... & ... \\
28 & 39 & 50 & 61 & 72 & ... & ...  \\
33 & 46 & 59 & 72 & 85 & ... & ...  \\ 
38 & 53 & 68 & 83 & 98 & ... & ...  \\
... & ... & ... & ... & ... & ... & ...
\end{pmatrix}
\end{equation} 
This is a symmetric matrix, and we observe \(AP\)s can be formed by taking elements from any particular row (or column). The same thing can be observed as well in the matrices of \(p\) and \(k\) values. One can think of characterising all \(AP\)s formed by taking elements from the matrix. \\
For example, if there are two or more terms in the same row (column), then we have infinitely many elements in the same row (column). This is easy to see; if \(a\) and \(b\) are in same row, \(b=a+nd\) for some \(n\) (\(d\) is the common difference), and then all the elements \(a+mnd\) will be in the same row (column). This occurrence of infinitely many elements with the same number of onsets can be quite useful in constructing musical pieces. One can further analyse what will \(AP\)s look like if all the elements are in different rows (columns). This can be done by considering paths; one can go from one term to the next by some finite vertical and horizontal steps, we call this a path. By using the fact that the path sum should remain the same in every step, and how increments in the steps in both the directions occur, one can estimate the locations of the terms of the \(AP\), in turn estimating \((p,k)\) values.

We can also draw an interesting observation about the rhythms using Dirichlet's result on prime numbers in an AP, according to which all the \(AP\)s in the column (or rows) of the solution matrix of \(k\) and \(n\) values contain infinitely many primes. In other words, in any \(AP\) of the solution matrices, infinitely many rhythms exist whose number of time intervals (or total number of rests) is a prime. This result is quite interesting from the point of view of Indian Classical music. Various theoretical models of musical compositions can be created based on this result.  \\

We now state a result connecting any arbitrary sequence of Euclidean rhythms in which the \(V(p,k)\)s form an \(AP\). 
\begin{theorem}
    Let \(\{a_n\}\) be a sequence of arbitrary Euclidean rhythms such that \(\{V(p_n,k_n)\}\) forms an arithmetic progression satisfying either of the following conditions:
    \begin{enumerate}
        \item there exist some number of rests between the last onset and the first onset (of the next cycle) in \(a_1\)
        \item the common difference of \(\{V(p_n,k_n)\}\) is not a multiple of 10
    \end{enumerate}
    Then \(\{a_n\}\) contains infinitely many Euclidean rhythms with palindromic rests \cite{ap}. 
    \end{theorem}

    \begin{proof}
    It can be easily proved that we cannot find any Euclidean rhythm with palindromic rests if both conditions do not hold for any sequence \(\{a_n\}\). Let either of them hold. Let \(d\) be the common difference of \(\{V(p_n,k_n)\}\), \(10|d\) (i.e., the second condition does not hold, and the first condition holds). Then by the first condition, \(10 \not| \ V(p_1,k_1)\). Let \(V(p_1,k_1)=[v_qv_{q-1}...v_1v_0]\); we define \(V'(p_1,k_1)=[v_0v_1...v_{q-1}v_q]\). Let \(d=2^m5^n\), then for any \(l=max\{q,m,n\}+r \ (r \in \mathbb{Z})\) the elements \(V(p_1,k_1)+10^l V'(p_1,k_1) \in \{V(p_n,k_n)\}\) and are of the form \[[v_0v_1...v_{q-1}v_q0....0v_qv_{q-1}...v_1v_0]\] which in turn are either of the form \([(m+1)m(m+1)m....(m+1)]\) or  \([(m+1)m(m+1)(m+1)(m+1)m(m+1)....(m+1)m(m+1)]\) and thus the corresponding rhythms are Euclidean rhythms with palindromic rests in \(\{a_n\}\).
    Now let \(10 \not| \ d\), \(d=2^m5^n c\). We can show that there exists \(e = max\{q,m,n\}+ r \ (r \in \ \mathbb{Z})\) such that \(10^e \equiv 1 \ (mod \ c)\). If there exist some number of rests between the last onset and the first onset (of the next cycle) in \(a_1\), there exist \(s \in \{0,1,...c-1\}\) such that \(V'(p_1,k_1). 10^{e-q} + s \equiv 0 \ (mod \ c)\). Consider \[x=V'(p_1,k_1). 10^{es+e-q}+10^{es}+10^{(s-1)e}+..+10^e+V(p_1,k_1)\] 
    \[x-V(p_1,k_1)=V'(p_1,k_1).10^{es+e-q}+10^{es}+10^{(s-1)e}+..+10^e\] \[\equiv V'(p_1,k_1).10^{se}. 10^{e-q}+ \underbrace{1+1+...+1}_{s} \ (mod \ c)\]
    Thus, \[x-V(p_1,k_1) \equiv V'(p_1,k_1). 10^{e-q} + s \equiv 0 \ (mod \ c)\]
    So, \(c |x-V(p_1,k_1)\). We already know that \(2^m, 5^n |x-V(p_1,k_1)\); putting them together, we get \(d |x-V(p_1,k_1)\) i.e., \(x \in \{V(p_n,k_n)\}\). Also, \[x=[v_0 v_1 ... v_{q-1}v_q \underbrace{0..0}_{e-q-1}1 0..010..010..01 0..01 \underbrace{0..0}_{e-q-1}v_q v_{q-1}...v_1v_0]\] so \(x\) is a palindrome. Also, infinitely many such \(x\) are present in \(\{V(p_n,k_n)\}\) since \(10^{ze} \equiv 1 \ (mod \ c)\) for any \(z \in \mathbb{N}\). 
    Note that if there are no rests between the last onset and first onset (of the next cycle) in \(a_1\), then we consider \(a_2\) and carry on the same steps. Hence infinitely many \(V(p_i,k_i)\)s are palindromes, and the corresponding rhythms are Euclidean rhythms with palindromic rests. Therefore, we have proven our statement.
    \end{proof} 

We can generalise this result slightly to any arbitrary \(AP\) of rhythms. If \(\{a_n\}\) is a sequence of arbitrary rhythms such that \(\{V(p_n,k_n)\}\) forms an arithmetic progression satisfying either of the two conditions, then \(\{a_n\}\) contains infinitely many rhythms with palindromic rests. Now, these rhythms may not be Euclidean, but we can think of them as a series of operations on Euclidean rhythms. For example, if we consider using only the first two operations defined in Section 8, and if they are in \(\tau_1\), we can write them as a series of operations on Euclidean rhythms, or some rotation of an Euclidean rhythm, or some broken-down Euclidean rhythms and a residue rhythm (a rhythm that is neither of the three)! But this cannot be done in all cases; we need some extra conditions.  
\begin{notes}
    \item This whole process can also be done with other operations in other ways. 
\end{notes}
We now think about how to write any rhythm with palindromic rests, as in the series we have stated. We think of the most natural constructions; we first write the rhythm as an operation of an Euclidean rhythm and any other rhythm. But for that, we need two conditions: there has to be infinitely many rhythms with an odd number of palindromic rests, and all the digits in the palindrome number should be greater than, equal to, or less than the first digit (the last digit). Now we can write the other rhythm as an operation of an Euclidean rhythm and any other rhythm and continue this process accordingly. Thus, we get the series. Notice that broken-down Euclidean rhythms or rotated euclidean rhythms are only required towards the end of the process, and condition 2 is required to hold in every step. The residue rhythm can be further removed from this series by suitable further conditions, and those can be studied accordingly. Analysis of this kind can also be done similarly for elements of \(\tau_2\). We illustrate this with two rhythms:
\begin{itemize} 
    \item \(V_1(p,n)=4976794\) can be written as \[V(7,25)+V(5,28)-V(3,5) \ (rotated)-11\] \([\times \cdot \times \cdot \times]\) is the residue rhythm.  \\
    \item \(V_2(p,n)=2953592\) can be written as \[V(7,11)+V(5,38)-434 \ (broken \ Euclidean \ rhythm) -3\] \([\times \cdot \cdot \cdot \times]\) is the residue rhythm. 
\end{itemize}\
\begin{theorem}
     Let \(\{a_n\}\) be a sequence of arbitrary rhythms such that \(\{V(p_n,k_n)\}\) forms an arithmetic progression satisfying either of the following conditions:
    \begin{itemize}
        \item there exist some number of rests between the last onset and the first onset (of the next cycle) in \(a_1\)
        \item the common difference of \(\{V(p_n,k_n)\}\) is not a multiple of 10
    \end{itemize}
    Then \(\{a_n\}\) contains infinitely many rhythms with palindromic rests. 
    Furthermore, if the following conditions apply:
   \begin{enumerate}
    \item there are infinitely many rhythms with an odd number of palindromic rests
    \item all the digits in the palindrome number should be either greater than, equal to, or less than the first digit (last digit)   
\end{enumerate}
are also satisfied, then we can write these rhythms with palindromic rests as a series of operations of Euclidean rhythms, or some rotation of an Euclidean rhythm, or some broken down Euclidean rhythms and a residue rhythm (a rhythm that is neither of the three). 
\end{theorem}

\section{Applications in Indian Classical Music}
We present some applications of Euclidean rhythms with palindromic rests in Indian Classical music. 
\begin{center}
    \includegraphics[width=15.9cm]{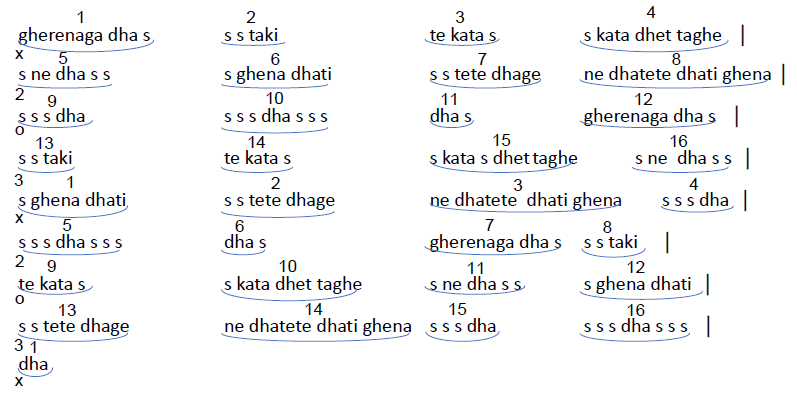}
\end{center}
This is a \(tehai\) in \(teen \ taal\) (the \(taal\) of \(16\) beats, following the division \(4/4/4/4\)) following \(chatusra \ jaati\). This comes to the \(sam\) on the \(33\)rd beat, in other words, this \(tehai\) comes to the \(sam\) after \(2\) cycles of \(16\) beats. Each \(s\) represents a rest, which is of length \(\frac{1}{4}\). The Euclidean rhythm with \(V(p,k)=32323\) has been used here. Although the pattern of rests is easy to observe, we see that a set of onsets is present between these rests rather than a single onset, which should not be the case! However, theoretically, each of these sets of onsets gets reduced to a single onset at some very high \(laya\), so all of these are Euclidean rhythms at some \(laya\) or tempo! This is a very beautiful concept, and it shows the importance of tempo in any musical composition. This concept will be used to show the application of Euclidean rhythms in all the other compositions as well. In all the rest of the compositions, only one part of the \(tehai\) has been written; the mathematics have been explained accordingly.
\begin{center}
    \includegraphics[width=15cm]{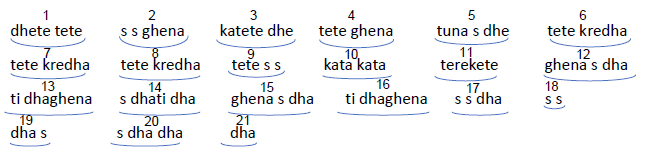}
\end{center}
This \(tehai\) follows \(chatusra \ jaati\) and is in \(teen \ taal\). The length of each rest is \(\frac{1}{4}\). Euclidean rhythm with \(V(p,k)=21212\) is used here. There is a gap of \(1\) length between each part, so the length of the \(tehai=21 \times 3+2=65=64+1\). This shows that the \(tehai\) comes back to \(sam\) after \(4\) cycles. 
\begin{center}
    \includegraphics[width=15.3
    cm]{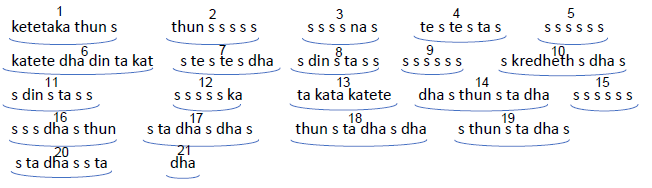}
\end{center}
This \(tehai\) follows \(tisra \ jaati\) and is in \(teen \ taal\). The length of each rest is \(\frac{1}{6}\). Euclidean rhythm with \(V(p,k)=43434\) is used here. There is a gap of \(1\) length between each part, so the length of the \(tehai=21 \times 3+2=65=64+1\). This shows that the \(tehai\) comes back to \(sam\) after \(4\) cycles. 
\begin{center}
    \includegraphics[width=15.9cm]{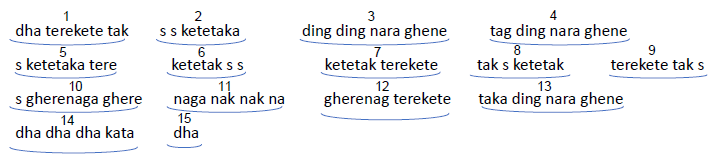}
\end{center}
This \(tehai\) follows \(chatusra \ jaati\) and is in \(pancham \ sawari \ taal\) (the \(taal\) of 15 beats, following the division \(4/3/5/3\)). Length of each rest is \(\frac{1}{4}\). Euclidean rhythm with \(V(p,k)=21212\) is used here. There are \(2\) rests between each part, each of length \(\frac{1}{4}\). Therefore, total length of the \(tehai=15 \times 3+\frac{1}{2} \times 2=46=45+1\). This shows that the \(tehai\) arrives at the \(sam\) after \(3\) cycles. 
\begin{center}
    \includegraphics[width=15.8cm]{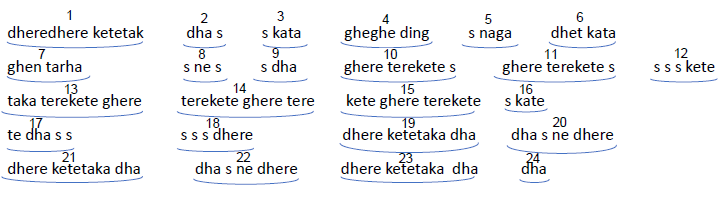}
\end{center}
This \(tehai\) follows \(chatusra \ jaati\) and is in \(jhaap \ taal\) (the \(taal\) of 10 beats, following the division \(2/3/2/3\)). The length of each rest is \(\frac{1}{4}\). Euclidean rhythm with \(V(p,k)=2121212\) is used here. There are no rests between each part of the \(tehai\), \(dha\) of the \(24\)th beat takes \(\frac{1}{2}\) length, total length of \(tehai=23.5 \times 3=70+\frac{1}{2}\) and so it comes to the \(sam\)
after \(7\) cycles.
\begin{center}
    \includegraphics[width=15.7cm]{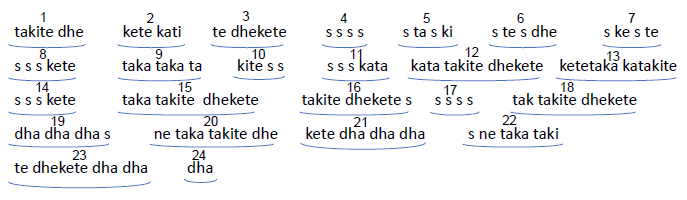}
\end{center}
This \(tehai\) follows \(chatusra \ jaati\) and is in \(dhamar \ taal\) (the \(taal\) of \(14\) beats, following the division \(5/2/3/4\)). The length of each rest is \(\frac{1}{4}\). Rhythm with \(V(p,k)\) representation \(53535\), which can in turn be generated by the addition operation we defined earlier, is used here. The operation is between the Euclidean rhythms with \(V(p,k)=21212\) and \(32323\). There are no rests between each part of the \(tehai\), the \(dha\) of the \(24\)th beat takes \(\frac{1}{2}\). Therefore, total length of the \(tehai=23.5 \times 3=70.5=70+\frac{1}{2}\), so it comes to \(sam\) after \(5\) cycles.
\begin{center}
    \includegraphics[width=15.9cm]{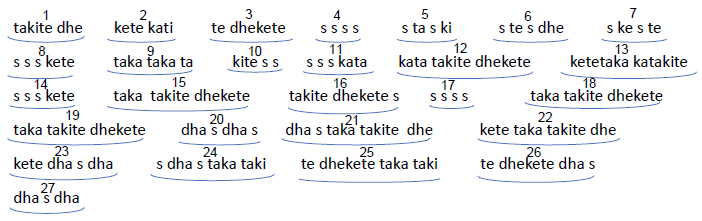}
\end{center}
This \(tehai\) is a different variation of the previous one in \(teen \ taal\) following \(chatusra \ jaati\), length of each rest being \(\frac{1}{4}\). Rhythm with \(V(p,k)\) representation \(53535\), which can in turn be generated by addition operation, is used here. This \(tehai\) is \(bedaam\), so there are no rests between each part of the \(tehai\). \(dha \ s \ dha \) of the \(27\)th beat is of length \(\frac{3}{4}\). As a result, the total length of the \(tehai=26.75 \times 3=80.25=80+\frac{1}{4}\). Thus the \(tehai\) comes to the \(sam\) after \(5\) cycles. 
\begin{center}
    \includegraphics[width=14.4cm]{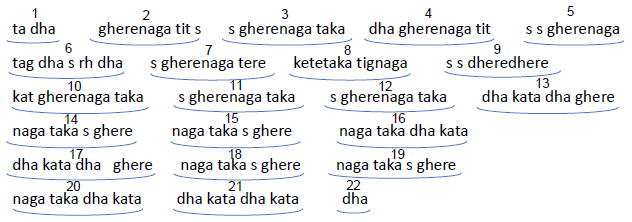}
\end{center}
This is a \(tehai\) in \(teen \ taal\) following \(chatusra \ jaati\), length of each rest being \(\frac{1}{4}\). The euclidean rhythm \([\times \cdot \cdot \times \cdot \cdot \times \cdot \times \cdot \cdot \times \cdot \times ] \) which is homometric to the rhythm with \(V(p,k)=21212\) is used here. This \(tehai\) is \(bedaam\), so there are no rests between each part of the \(tehai\). The \(dha\) of the \(22\)nd beat takes \(\frac{1}{2}\) length, the next part starts from then on. Therefore, the total length of the \(tehai=21.5 \times 3=64.5=64+\frac{1}{2}\). This shows that the \(tehai\) comes to the \(sam\) after \(4\) cycles.
\section{Concluding remarks}
The rhythms we worked on in this paper have a very simple structure in general. One can think of analysing them using various other mathematical methods. However, while analysing rhythms, there is always a restriction to the amount of mathematical abstractness one uses in analysing them. Any such analysis should always have some reflection on the practical demonstration of such rhythms. In a sense, one can say that many areas of theoretical research on rhythms are in some way restricted by how mathematically advanced musical compositions are; one needs to make advancements in the types of practises of music in order to pave the way for new theoretical interpretations of rhythms.\\
Based on our analysis and intuition, we believe that the analysis of \(AP\)s in Section 9 can be strengthened; we believe that any finite-length \(AP\) is totally contained in a particular row (or column). However, at present, we are unable to provide a definitive proof for this claim. We can analyse the rhythms further by adopting tools used by Toussaint in \cite{Toussaint}. Since every rhythm can be reinterpreted as a scale, the Euclidean rhythms with palindromic rests can be further analysed by drawing ideas from Clough and Douthett's work \cite{even}.
\section*{Acknowledgement}
I would like to thank Mr. Shibam Mondal and Mr. Ankit Saha for their ideas and help throughout. I would also like to express my deep gratitude to Professor Jean Paul Allouche for his suggestions. I would also like to thank all the referees for their precious corrections and the mathematicians, professors, musicologists, and musicians who will review this paper.

\end{document}